\documentclass[a4paper]{amsart}
\usepackage{amsmath,amsthm,amssymb,latexsym,epic,bbm,comment,color}
\usepackage{graphicx,enumerate,stmaryrd}
\usepackage[all,2cell]{xy}
\xyoption{2cell}

\usepackage{tikz-cd}

\newcommand{\Z}{\mathbb Z}

\newcommand{\R}{\mathbb R}
\newcommand{\K}{\mathbb K}
\newcommand{\C}{\mathbb C}

\newcommand{\pp}{\mathsf{p}}

\newcommand{\mcF}{\mathcal F}

\newcommand{\mcM}{\mathcal M}
\newcommand{\mcN}{\mathcal N}
\newcommand{\mcO}{\mathcal O}
\newcommand{\mcP}{\mathcal P}

\newcommand{\mcU}{\mathcal U}
\newcommand{\mcV}{\mathcal V}

\newcommand{\mcS}{\mathcal S}

\newcommand{\mcW}{\mathcal W}

\newcommand{\F}{\mathrm{F}}

\newcommand{\id}{\mathrm{id}}

\newcommand{\Ker}{\operatorname{Ker}}

\newtheorem{theorem}{Theorem}
\newtheorem{lemma}[theorem]{Lemma}
\newtheorem{definition}[theorem]{Definition}
\newtheorem{corollary}[theorem]{Corollary}
\newtheorem{proposition}[theorem]{Proposition}

\newtheorem{example}[theorem]{Example}

\newtheorem{remark}[theorem]{Remark}

\begin{document}
	\title[$H$-covering of a supermanifold]
	{$H$-covering of a supermanifold}

	\author{Fernando A.Z. Santamaria and Elizaveta Vishnyakova}
	
	\begin{abstract}

    We develop the theory of $H$-graded manifolds
		for any finitely generated abelian group, using tools from representation theory. Furthermore, we introduce and investigate the notion of $H$-graded coverings of supermanifolds in the case where $H$ is a finite abelian group.
			\end{abstract}

	\maketitle

\section{Introduction}

The category of $A$-graded manifolds, where $A$ is an abelian group or an abelian monoid, was studied by several authors; see, for instance, \cite{Fei, Ji, KotovSalnikov, Vysoky}, see also \cite{AKT, BHM, Ker, Ly, RW}. An $A$-graded manifold is a ringed space, which is locally isomorphic to a graded domain $(S, \mcO_{\mcS})$. An element $f$ in the structure sheaf $\mcO_{\mcS}$ of an $A$-graded domain is usually assumed to be a series in colored coordinates,  that is in coordinates of weight $a\in A\setminus \{e\}$. An alternative approach, based on the use of a grading operator, is presented in \cite{Grab}, see also the references therein for further literature.

In this paper, we consider $H$-graded domains, where $H$ is a finitely generated abelian group, where elements of the structure sheaf are smooth or holomorphic functions (not series) with respect to local graded coordinates. In more detail, let $G$ be a finitely generated abelian group and $H=\operatorname{Ch}(G)$ be its dual group, that is, the group of complex characters of $G$. First, we develop the theory of graded manifolds for a finite abelian group $G$.  In this case, we have $G\simeq H$. 
The case of a finite abelian group is very special due to the fact that any representation of such a group may be decomposed into a direct sum of irreducible representations. Our definition of a $H$-graded domain is as follows. Let a homomorphism $\phi:H\to \Z_2$ be fixed.  We start with a representation of $G$ in a vector space $V$. Then $V$ is naturally $H$- and  $\Z_2$-graded in the following way
$$
V= \bigoplus_{\chi\in H} V_{\chi}, \quad V_{\bar i} = \bigoplus_{\phi(\chi) =\bar i} V_{\chi}. 
$$
Using this data we define a $H$-graded domain $(U,\mcO)$, where $U\subset V_{e}$, where $e\in H$ is the identity, and $\mcO$ is the inverse image sheaf of the sheaf of smooth or holomorphic superfunctions on $V$, see the main text for details.  Using the representation theory for finite groups we decompose the structure sheaf $\mcO$ into $H$-homogeneous components
$$
\mcO = \bigoplus_{\chi\in H} \mcO_{\chi}.
$$
This defines a $H$-gradation in the structure sheaf. We call a function $F\in \mcO$ homogeneous of weight $\chi\in H$ if 
$$
g\cdot F = \chi(g) F.
$$
In the last section, we generalize this approach to any finitely generated abelian group. In this case, we do not have a decomposition of the structure sheaf into homogeneous components. However, since the notion of a homogeneous function is defined, we can define a homogeneous morphism of $H$-graded domains, and hence a $H$-graded manifold. Our approach is related to, but distinct from, that of \cite{Grab}; see the final section for details. Our results can also be applied to construct graded group objects for graded Lie (super)algebras; see \cite{AKY, DEl, El}.

Another result of this paper is as follows. In \cite{Vicovering}, the notion of a $\mathbb{Z}$-graded covering space for a supermanifold was introduced, and later extended in \cite{Vimult} to the case of graded manifolds. The $\Z$-graded covering of a supermanifold can be seen as the space of $\mcN$-paths in a supermanifold, where $\mcN$ is a $\Z$-graded manifold. Such structures also appear in classical geometry: $m$-jet space of a manifold $X$, that is $Spec(\K[t]/ t^{m+1})$-paths in a (algebraic) manifold $X$, i.e. morphisms of $Spec(\K[t]/ t^{m+1})$ to $X$; the arc space of $X$, that is $Spec(\K[[t]])$-paths in a (algebraic) manifold $X$, where $\K[[t]]$ is the ring of formal series in $t$, see  \cite[Definition 2.1]{Ishii}.  One also considers the loop space of $X$, that is $\K((t))$-points in $X$, where $\K((t))$ is the ring of formal Laurent series in $t$, where it allows finitely many negative powers of $t$, but infinitely many nonnegative ones, see \cite[Proposition 1.5]{AM} and \cite{KV}. The jet space and arc space of $X$ satisfy a similar universal property as considered in this paper, see \cite[Theorem 1.1]{AM}, see also \cite[Proposition 1.5]{AM} for the loop space in the affine case.

The $H$-covering space constructed in this paper classifies the space of $\mcN$-paths in a manifold or supermanifold, where $\mcN$ is a $H$-graded manifold for a finite abelian group $H$. In other words, $H$-covering space classifies morphisms from a $H$-graded manifold to a supermanifold.  In more detail, for any supermanifold $\mcM$ we construct a $H$-graded manifold $\mcP$ with a covering projection $\pp: \mcP \to \mcM$, which satisfies the universal property: any $\Z_2$-graded morphism of a $H$-graded manifold to the supermanifold $\mcM$ can be uniquely  lifted to a $H$-graded morphism $\mcN \to \mcP$. We have a similar universal property for a topological covering of a usual manifold. This also explains the name for the 'covering space'. For instance, a graded covering $\mcP$ is unique up to isomorphism, and any morphism $\phi: \mcN\to \mcM$, where $\mcN$ is a $H$-graded manifold, can be factored through the covering projection $\pp: \mcP \to \mcM$. We expect that important graded $H$-spaces can be interpreted as $H$-coverings of some (super)manifolds.

\bigskip

\textbf{Acknowledgments:} F.S. was supported by FAPEMIG. 
E.V. was partially  supported 
by FAPEMIG, grant APQ-01999-18, FAPEMIG grant RED-00133-21 - Rede Mineira de Matem\'{a}tica, CNPq grant 402320/2023-9. The authors thank Dmitry Timashev for useful discussions and the anonymous referee for helpful comments.

\section{Preliminaries}

\subsection{Characters of finite abelian groups}\label{sec Characters of finite abelian group}
We work over the field of complex numbers $\C$ and for $G=(\Z_2)^n$ over $\R$ or $\C$. Let $G$ be a finite abelian group. Denote by $\Z_{q}$ the cyclic group of order $q\in \mathbb N$.  The structure theorem for finite abelian groups states that every abelian group is isomorphic to a direct product of cyclic
groups. In other words,  
$$
G\simeq \Z_{q_1}\times \cdots \times \Z_{q_t},
$$
where the numbers $q_1, \ldots, q_t$ are powers of prime numbers (not necessarily distinct); see \cite[Chapter 2]{Kurzweil}. For an introduction to the character theory of finite groups, we refer to \cite{Isaacs}.

A group homomorphism $\chi: G\to \C^*=\C\setminus \{0\}$ is called a {\it character of $G$}. If $\chi_1$, $\chi_2$ are two characters of $G$, we can define their product $\chi_1\cdot \chi_2$ by the following formula
$$
[\chi_1\cdot \chi_2] (g) =\chi_1(g)\cdot \chi_2(g), \quad  g\in G.
$$ 
The set of all characters $\operatorname{Ch}(G)$ is an abelian group with respect to the introduced operation.  For a cyclic group $\Z_{q}$ with a generator $u$ we have $\chi(u)^q=1$, so $\chi(u)$ is a root of unity. The group $\operatorname{Ch}(\Z_{q})$ can be identified with the roots of unity group
 $$
 \Z_{q} \simeq \{e^{2\pi i\frac{k}{q}}\,\,|\,\, k=0,1,\ldots ,q-1 \}\subset \C.
 $$
Indeed, for $\Z_{q} = \langle u\rangle$, the map $\chi\mapsto \chi(u)$ defines an isomorphism 
$$ 
\operatorname{Ch}(\Z_{q}) \to \{e^{2\pi i\frac{k}{q}}\,\,|\,\, k=0,1,\ldots, q-1 \}.
$$

We also have $\Z_{q} \simeq \operatorname{Ch}(\Z_{q})$. Furthermore, if $G=G_1\times G_2$, then any character $\chi: G \to \C^*$ is a product of characters $\chi_i:G_i\to \C^*$, where $\chi_i= \chi|_{G_i}$. In other words, 
$$
\operatorname{Ch}(G_1\times G_2) \simeq \operatorname{Ch}(G_1)\times \operatorname{Ch}(G_2).
$$ 
In particular for an abelian group $G$ we have
$$
\operatorname{Ch}(G) = \operatorname{Ch}(\Z_{q_1}\times \cdots \times  \Z_{q_t}) \simeq \operatorname{Ch}(\Z_{q_1})\times \cdots \times  \operatorname{Ch}(\Z_{q_t}) \simeq G.  
$$
This isomorphism is not unique.

\begin{remark}\label{rem Z_2 charcters are real}
    In the case $G=(\Z_2)^n$ any character is pure real. 
\end{remark}

\subsection{$\operatorname{Ch}(G)$-graded vector superspaces}\label{sec G-graded vector superspaces}

 Let $G$ be a finite abelian group and $H:=\operatorname{Ch}(G)$ be its character group equipped with a surjective homomorphism 
 $$
 H\to \mathbb Z_2=\{\bar 0,\bar 1\}, \quad  \chi \mapsto |\chi|.
 $$
  We will call the image of $\chi\in H$ by our homomorphism the parity of $\chi$ and denote this image by $|\chi|$.  
 Now consider any complex representation of $G$ in a non-necessary finite-dimensional vector space $W$. In the case $G=\Z_2^n$ we can assume that $W$ is real. 

 \begin{theorem}\label{theor decomposition superformula}
     There is the following decomposition \begin{equation}\label{eq H-decomposition of V}
 W = \bigoplus_{\chi\in \operatorname{Ch}(G)} W_{\chi},\quad W_{\chi} = \{v\in W\,\,|\,\,  g\cdot v = \chi(g) v \, \}.
 \end{equation}
 \end{theorem}
 
 \begin{proof}
If $v\in W$ is any element, then for $v_{\chi}\in W$ defined by the following formula
 \begin{equation}\label{eq H-decomposition of any element v formula}
v_{\chi} = \frac{1}{|G|} \sum_{g\in G} \chi (g) (g^{-1}\cdot v)
\end{equation}
we have $h\cdot v_{\chi} = \chi(h) v$ for any $h\in G$. In other words, $v_{\chi}\in W_{\chi}$. 
Furthermore, we have
$$
v=\sum_{\chi\in H} v_{\chi}.
$$
Indeed,  for any $g\in G\setminus \{e\}$ there exists $\chi_0\in H$ such that $\chi_0(g)\ne 1$. Indeed,  assuming the opposite, we find that there exists $g\in G\setminus \{e\}$ such that for any character $\chi \in H$ we have $\chi(g)=1$. Therefore, 
$$
G \simeq\operatorname{Ch}(G) \simeq \operatorname{Ch}(G/ G')\simeq G/ G',
$$ 
where $G'$ is the non-trivial subgroup generated by $g$.  This is a contradiction.

Hence, we have
\begin{align*}
    \sum_{\chi\in H} \chi(g) = \sum_{\chi\in H} [\chi_0\cdot\chi](g) = \chi_0(g) \sum_{\chi\in H} \chi(g).
\end{align*}
Therefore, $\sum\limits_{\chi\in H} \chi(g) =0$, where $g\ne e$, and clearly $\sum\limits_{\chi\in H} \chi(e) =|G|$. Summing up, 
\begin{align*}
\sum_{\chi\in H} v_{\chi} = \sum_{\chi\in H} \Big(\frac{1}{|G|} \sum_{g\in G} \chi (g) (g^{-1}\cdot v)\Big) = \frac{1}{|G|}  \sum_{g\in G} \Big( \sum_{\chi\in H} \chi (g)  \Big) (g^{-1}\cdot v) =v.
\end{align*}
\end{proof}

 Now we apply the result of Theorem \ref{theor decomposition superformula} to the finite-dimensional case. Consider a finite-dimensional representation $V$ of $G$.  Later we will use $V$ to construct a graded domain, which is assumed to be finite-dimensional.  
 The decomposition (\ref{eq H-decomposition of V}) leads to the following decomposition 
 \begin{align*}
      V = \bigoplus_{\chi\in \operatorname{Ch}(G)} V_{\chi},\quad V_{\chi} = \{v\in V\,\,|\,\,  g\cdot v = \chi(g) v \, \},
 \end{align*}
 and to the following decomposition 
\begin{equation}\label{eq G-graded superspace}
V=V_{\overline{0} } \oplus V_{\overline{1}},\quad V_{\overline{i} } = \bigoplus_{|\chi|=\bar i} V_{\chi}, \quad \bar i\in \{\bar 0,\bar 1\}.
\end{equation}
 An element $v\in V_{\chi}$ is called {\it homogeneous} of weight $\chi$ and of parity $|\chi|$. In more detail, $v\in V$ is called homogeneous of weight $\chi$ if for any $g\in G$ the following holds 
 \begin{equation}\label{eq action on V by characters}
g\cdot v =  \chi(g) v, \quad g\in G.
\end{equation}
Clearly, $\dim V_{\chi} < \infty$ for any $\chi\in H$. 
 
\begin{remark}\label{rem action of G using homogeneous elements}
    If $v=\sum\limits_{\chi\in H} v_{\chi}$, see (\ref{eq H-decomposition of any element v formula}), the action of $G$ on $V$ is also given by the formula (\ref{eq action on V by characters}). In other words,
    $$
    g\cdot v = g\cdot \sum_{\chi\in H} v_{\chi} = \sum_{\chi\in H} \chi(g) v_{\chi}. 
    $$
This means that any complex representation $V$ of $G$ has this form and, in particular, has a homogeneous basis. 
\end{remark}

\begin{example}\label{ex Klein group, definition of K-graded V}
Let $K=\lbrace e, a, b, ab \rbrace$ be the Klein group.  Recall that $K$ is abelian and is defined by the following relations
$$
a^2=b^2=(ab)^2=e.
$$ 
Let us describe the group of characters $H=\operatorname{Ch}(K)$, see Section \ref{sec Characters of finite abelian group}.  Define characters $\chi_{a}, \chi_{b}: K\to \C^*$  by the following formulas
\begin{align*}
\chi_{a} (a)&= -1,\quad \chi_{a} (b) =1;\quad \chi_{b} (a)= 1,\quad \chi_{b} (b) =-1;\\
&\chi_{e} (a) = \chi_{e} (b) =1;\quad \chi_{ab}: = \chi_{a}\cdot \chi_{b}.
\end{align*}

Assume that $|\chi_{a}|=|\chi_{b}|=\bar 1$, $|\chi_{e}|= |\chi_{ab}| =\bar 0$.  Let $V$ be a finite-dimensional $H$-graded vector superspace. That is  
$$
V=V_{\chi_{e}}  \oplus V_{\chi_{a}} \oplus V_{\chi_{b}}\oplus V_{\chi_{ab}},\quad V_{\overline{0} } = V_{\chi_{e}}\oplus V_{\chi_{ab}},\quad  V_{\overline{1} } = V_{\chi_{a}}\oplus V_{\chi_{b}}.
$$
According to Remark \ref{rem action of G using homogeneous elements}, in this case a $K$-action on $V$ is given by Formula (\ref{eq action on V by characters}). More precisely we have
\begin{align*}
g\cdot v_{\chi_{c}}= \chi_{c}(g) v_{\chi_{c}} , \quad g\in K,\quad v_{\chi_{c}}\in V_{\chi_{c}}.
\end{align*}
We see that the map $c\mapsto \chi_c$ defines a homomorphism $K\to \operatorname{Ch}(K)$. Using this homomorphism, $V$ can be regarded as a $G$-graded vector space. However, such a homomorphism is clearly not unique. For example, another homomorphism is defined by $a \mapsto \chi_b$,  $b \mapsto \chi_a$, which leads to another $G$-grading on $V$. 
\end{example}

We can naturally define an action of $G$ on $V^*$ in the following way: 
$$
[g\cdot f](v) := f(g^{-1}\cdot v), \quad f\in V^*.  
$$
The vector space $V^* = \bigoplus\limits_{\chi\in H} V^*_{\chi}$ is $H$-graded. An element $f\in V^*$ is called homogeneous of weight $\chi\in H$, if $g\cdot f = \chi(g)f$ for any $g\in G$.  We can naturally identify  $V^*_{\chi}$ with  $(V_{\chi^{-1}})^*$.

\begin{remark}\label{rem V may be pure real G=Z_2}
    In the case $G=\Z_2^n$ we can consider a real representation $V$ of $G$, since all characters $\chi: G \to \C^*$ are real in this case. 
\end{remark}

\subsection{$\operatorname{Ch}(G)$-decomposition of the sheaf of smooth or holomorphic functions on $V$}\label{sec G-decomposition of F}

Let $G\simeq \Z_{q_1}\times \cdots \times  \Z_{q_t}$ be a finite abelian group that acts in a complex finite-dimensional vector space $V$. As in Section \ref{sec G-graded vector superspaces}, $H= \operatorname{Ch}(G)$ is the dual group of $G$ with fixed homomorphism $H\to \Z_2$, $\chi\mapsto |\chi|$, and $V=V_{\bar 0}\oplus V_{\bar 1}$ is an $H$-graded vector superspace.   Clearly, any homogeneous subspace $V_{\chi}$ is $G$-invariant.  Since $V$ is a vector superspace, we can consider the sheaf of superfunctions on $V$. By definition a superfunction $F$ on $V$ is an element in $\mcF\otimes_{\C} \bigwedge(V^*_{\bar 1})$, where $\mcF$ is the sheaf of smooth {\bf complex-valued} or holomorphic functions on $V_{\bar 0}$. In more detail, since $V_{\bar 0}$ is a complex vector space, we can identify it with $\C^{n}$ for some $n$. Then a smooth function $F: \C^{n} \to \C$ is a complex-valued smooth function on $\C^{n} \simeq \R^{2n}$. The same is true for odd variables: if $\dim_{\C} V^*_{\bar 1}=m$, then a smooth function depends on $2m$ odd variables. In the case $G=\Z_2^n$, we also can consider a real representation $V$ and real-valued even and odd smooth functions.

If $U$ is a $G$-invariant open subset in $V_{\bar 0}$ we define an action on $\mcF(U)\otimes_{\C} \bigwedge(V^*_{\bar 1})$ by
\begin{align*}
g\cdot (f\otimes \zeta) = f\circ g^{-1} \otimes g\cdot \zeta, \quad f\in \mcF(U), \,\,\, \zeta \in    \bigwedge(V^*_{\bar 1}). 
\end{align*}

\begin{definition}\label{def superfunction if homogeneous with weight chi}
	We call a superfunction $F\in \mcF(U)\otimes_{\C} \bigwedge(V^*_{\bar 1})$, where $U$ is a $G$-invariant open set in $V_{\bar 0}$, homogeneous of weight $\chi\in  H=\operatorname{Ch}(G)$ if the following holds
	$$
	g\cdot F(v) =  \chi(g) F(v),\quad v\in V
	$$
 for any $g\in G$. 
\end{definition}

Let $U$ be a $G$-invariant open set in $V_{\bar 0}$ and let $F\in \mcF(U)\otimes_{\C} \bigwedge(V^*_{\bar 1})$. We can decompose $F$ into homogeneous components, see Theorem \ref{theor decomposition superformula}, in the following way
\begin{equation}\label{eq G-decomposition of F}
F=\sum_{\chi\in H} F_{\chi},\quad F_{\chi} = \frac{1}{|G|} \sum_{g\in G} \chi (g) [g^{-1}\cdot F].
\end{equation}
A standard computation shows that 
$$
h\cdot F_{\chi} = \chi(h) F_{\chi}\quad \text{for any $h\in G$}.
$$
The decomposition (\ref{eq G-decomposition of F}) leads to the following decomposition 
\begin{equation}
\mcF(U)\otimes_{\C} \bigwedge(V^*_{\bar 1}) = \bigoplus_{\chi \in H} (\mcF(U)\otimes_{\C} \bigwedge(V^*_{\bar 1}))_{\chi}. 
\end{equation}
Note that we can define such a decomposition only for $G$-invariant sets $U$.  

Definition \ref{def superfunction if homogeneous with weight chi} needs some comments. Indeed, now the parity of a homogeneous element $F\in \mcF(U)\otimes_{\C} \bigwedge(V^*_{\bar 1})$ of weight $\chi$ is induced by the parity of a character $|\chi|$ and by the degree $p$ of $\bigwedge^{p}(V^*_{\bar 1})$. In fact, both definitions coincide. It is sufficient to show the following proposition. 

\begin{proposition}
    Let $F\in \mcF(U)$ and let $\chi\in H$ with $|\chi|=\bar 1$. Then $F_{\chi}=0$. 
\end{proposition}

\begin{proof}
We need the following fact from character theory. Let $G_0$ be any subgroup in $G$. Then we get the following two exact sequences of groups
\begin{align*}
 e \to  G_0 \longrightarrow G \longrightarrow G/G_0 \to e, \quad
     e \to  \operatorname{Ch}(G/G_0) \longrightarrow \operatorname{Ch}(G) \longrightarrow \operatorname{Ch}(G_0) \to e.
\end{align*}
And such exact sequences are in bijection. Therefore for the exact sequence of groups of characters
\begin{align*}
    e \to  H_0 \longrightarrow H \longrightarrow \Z_2 \to e
\end{align*}
we can find the corresponding exact sequence of subgroups
\begin{align*}
   e \to  G_0 \longrightarrow G \longrightarrow G/G_0 \to e.
\end{align*}
such that $H_0= \operatorname{Ch}(G/G_0)$ and $\Z_2 \simeq \operatorname{Ch}(G_0)$. We can also describe the subgroup $G_0$  in the following way
 $$
 G_0 = \{ g\in G \,\, |\,\, \chi(g) =1 \,\,\, \text{for any}\,\,\chi\in H_0 \}.
 $$
Since $\operatorname{Ch}(G_0)\simeq G_0$, we have $G_0\simeq \Z_2$. Let $G_0=\{e,a\}$. 

Let us fix $\chi \in H$ with $|\chi|=\bar 1$. Since $|\chi|=\bar 1$, we have $\chi(a)=-1$. (Otherwise, $\chi\in \operatorname{Ch}(G/G_0) = H_0$.) The group $G_0$  acts on $G$. Hence $G$ is divided into $G_0$-orbits $\{g_i, g_i \cdot a\}$ for some $g_i\in G$, which do not intersect. We have 
$$
\chi(g_i) [(g_i)^{-1}\cdot F] + \chi(g_i \cdot a) [(g_i \cdot a)^{-1}\cdot F] = \chi(g_i) [(g_i)^{-1}\cdot F] (1-1) =0.
$$
  Here we used the fact $a \cdot F = F$ since $F$ depends only on coordinates of weight $\chi_0\in H_0$. We see that the sum (\ref{eq G-decomposition of F}) is zero for such $\chi$.  This implies that $F_{\chi}=0$.   
\end{proof}

\begin{remark}\label{rem two ways to count weight of monomial}

If $F_i$ is a homogeneous superfunction of weight $\chi_i$, $i=1,2$, then $F_1\cdot F_2$ is a homogeneous superfunction of weight $\chi_1\cdot \chi_2$. Furthermore, let $(x_i^{h_1})$, where $|h_1| = \bar 0$, and $(\xi_j^{h_2})$, where $|h_2|=\bar 1$, is a basis of $V^*_{h_1}$ and $V^*_{h_2}$, $h_1,h_2\in H$, respectively. By Definition \ref{def superfunction if homogeneous with weight chi}, any $x_i^{h_1}$ is a homogeneous superfunction of weight $h_1\in H$ and $\xi_j^{h_2}$ is a homogeneous superfunction of weight $h_2\in H$. Hence, the weight of a monomial 
	$$
	x_{i_1}^{g_1}\cdots x_{i_s}^{g_s} \cdot \xi_{j_1}^{h_1}\cdots \xi_{j_t}^{h_t}
	$$
	is the product $g_1 \cdots  g_s \cdot h_1 \cdots  h_t\in H$ of weights. 
 \end{remark}

We finish this section with an example. 

\begin{example}{\bf $K$-decomposition of the sheaf of superfunctions.}\label{ex K decomposition of F(U)}
	Consider the case of the Klein group $K$, see Example \ref{ex Klein group, definition of K-graded V}, with a representation of $K$ in $V$. Then $V$ is also a $H$-graded, where $H=\operatorname{Ch(K)}$ and $H$-grading is defined in Theorem \ref{theor decomposition superformula}, see Example \ref{ex Klein group, definition of K-graded V}. Since $K\simeq \Z_2^2$, we can assume that $V$ is real. Let us take $F\in \mcF(U)\otimes_{\R} \bigwedge(V^*_{\bar 1})$, where $U\subset V_{\bar 0}$ is $K$-invariant, and $\mcF$ is the sheaf of real-valued smooth functions in $V_{\bar 0}$.

	By Remark \ref{rem two ways to count weight of monomial}, we have for any $h\in H$
	$$
	[\mcF(U)\otimes_{\R} \bigwedge(V^*_{\bar 1})]_h = \bigoplus_{h_1\cdot h_2= h} \mcF(U)_{h_1} \otimes_{\R} \bigwedge(V^*_{\bar 1})_{h_2}. 
	$$
	Further elements from $\bigwedge(V^*_{\bar 1})$ are Grassmann polynomials, their weights are determined by weights of generators, see Remark \ref{rem two ways to count weight of monomial}. Let us decompose an element $F\in \mcF(U)$ into homogeneous components. We use notations and parity agreement from Example \ref{ex Klein group, definition of K-graded V}.

	Assume that $F(x,y )\in \mcF(U)$, where $x = (x_1, \cdots, x_{p})$ is a basis in $V_{\chi_{e}}$ and $y = (y_1, \cdots, y_{q})$ is a basis in $V_{\chi_{ab}}$. (Recall that both $\chi_{e}$ and $\chi_{ab}$ are assumed to be even.) Applying Formula (\ref{eq G-decomposition of F}), we get
	\begin{align*}
	F_{\chi_{a}} = &\frac{1}{4} \big(\chi_{a}(e) F(e\cdot x,e\cdot y) + \chi_{a}(a) F(a\cdot x,a\cdot y) +\\
	&\chi_{a}(b) F(b\cdot x,b\cdot y) + \chi_{a}(ab) F(ab\cdot x,ab\cdot y)\big)  =\\
	&\frac{1}{4}  \big( F(x,y) -  F(x,-y) + F(x,-y) - F(x,y)\big)=0. 
	\end{align*}
	Here $-x := (-x_1, \cdots, -x_{p})$ and the same notation for $y$. 
	Similarly, we have 
	\begin{align*}
	F_{\chi_{e}}= \frac12 (F(x,y) + F(x,-y)), \quad F_{\chi_{b}} = 0, \quad F_{\chi_{ab}} = \frac12 (F(x,y) - F(x,-y)). 
	\end{align*}
	Summing up,  we got the following decomposition of the superfunction $F$ into homogeneous components: 
	\begin{align*}
	F = \frac12 (F(x,y) + F(x,-y)) + \frac12 (F(x,y) - F(x,-y)). 
	\end{align*}
	This decomposition corresponds to the well-known decomposition of a function into even and odd parts. 
 \end{example}

\subsection{Taylor series of homogeneous components}\label{sec Taylor} We use the notation of Section \ref{sec G-decomposition of F}.
Let us study the Taylor series. We start with the following example. 

\medskip
{\bf Warning.}
Let $V=V_{\bar 0} = V_e\oplus V_{\chi}\oplus V_{\chi^2}$ be a $\Z_3$-graded vector space. (The choice of $\Z_3$ is not essential.) Let $\dim V_{\chi}=1$ and $(x)$ be the basis of $V_{\chi}$. Denote by $(y)$, where $y\in (V_{\chi})^*$ the dual co-basis. That is $y(x)=1$.
We note that one may try to write the Taylor series in following form 
$$
\sin z= z- \frac{z^3}{3!} + \frac{z^5}{5!} + \cdots $$ for $z\in V_{\chi}\simeq \C.$ This is wrong since $z^2$ is not defined and since for $a\in \Z_3\setminus\{e\}$ we get the following wrong equality
$$
\sin (a^{-1} \cdot z) =
a\cdot \sin z  \neq  a\cdot z- \frac{(a\cdot z)^3}{3!} + \frac{(a\cdot z)^5}{5!} + \cdots .
$$

 Consider the function $\sin \circ y : V_{\chi}\to \C$, given by $\sin\circ y(\alpha x) = \sin (\alpha)$, where $\alpha\in \C$. Then we can write the corresponding Taylor series in the following form
\begin{align*}
    \sin\circ y = y- \frac{y^3}{3!} + \frac{y^5}{5!} + \cdots .
\end{align*}
We see that for any $\alpha x$ the value of the function on the right-hand side and the left-hand side is the same.  We also see that the Taylor series of the composition $ \sin\circ y$  has the usual form. 

\medskip

 Choose a homogeneous basis $(x^{h^{-1}}_i)$ of $V_{h^{-1}}$ and the dual co-basis $(y^{h}_j)$ of $V^*_h\simeq(V_{h^{-1}})^*$, where $h\in H_0=\Ker \phi$. Let $F\circ (y^{h}_j)\in \mcF(U)$, where $U$ is a  $G$-invariant open set as above.  For simplicity, let us use the following notation: $z_1, \ldots, z_p$ for $(y^{h}_j)$, where $h\ne e$, and $t_1, \ldots, t_q$ for $(y^{e}_j)$. In addition, we put $t=(t_1, \ldots, t_q)$, $z=(z_1, \ldots, z_p)$, $n=(n_1,\ldots, n_p)$ and $z^n = z_1^{n_1} \cdots z_p^{n_p}$. Since our function is smooth or holomorphic, we can consider its Taylor series with respect to $z$ in a neighborhood of the point $(t, 0)$. We have
 \begin{align*}
     F\circ (t,z) \mapsto T_F(t,z) = \sum a_{n}(t)z^n,
 \end{align*}
where $a_{n}(t)$ are smooth or holomorphic functions in $t$.  Clearly, the $G$-action in $V$ induces the $G$-action on Taylor series.

\begin{remark}
    In the smooth category the Taylor series above is the Taylor series with respect to real variables. For example, for $F(z)$, were the weight of $z =x+iy$ is not $e$, our Taylor series is in fact the Taylor series in real coordinates $(x,y)$. Recall that we also can rewrite this Taylor series with respect to variables $z, \bar z= x-iy$. 
 \end{remark}

\begin{lemma}\label{lem T_(g F) =g T_F} 
We have $g\cdot T_F =  T_{g\cdot F}$ for any $g\in G$. 
\end{lemma}

\begin{proof}
    First of all we have
    \begin{align*}
        [g\cdot  F]\circ (t,z) (u,v) = F\circ (t,z) (g^{-1} u,g^{-1} v ) = F\circ (g\cdot t, g\cdot z) (u,v),
    \end{align*}
    for any point $(u,v)$ in a neighborhood of the point $(t, 0)$. This means that 
    $$
    [g\cdot  F]\circ (t,z) = F\circ (g\cdot t, g\cdot z). 
    $$
    Hence, 
    \begin{align*}
        T_{[g\cdot  F]\circ (t,z)}  = T_{F\circ (g\cdot t, g\cdot z)} = \sum a_{n}(t)(g \cdot z)^n = g\cdot T_{F\circ (t,z)}.
    \end{align*}
    Note that $g\cdot t=t$ for any $g\in G$.
\end{proof}

Now let $T_F$ be the Taylor series of $F$. We apply Formula (\ref{eq H-decomposition of any element v formula}) to $T_F$ to get its $H$-decomposition 
\begin{align*}
 T_{F} = \sum_{\chi\in H} (T_{F})_{\chi}, \quad    (T_{F})_{\chi} = \frac{1}{|G|} \sum\limits_{g\in G} \chi (g) \,\, [g^{-1}\cdot T_{F}].
\end{align*}

Lemma \ref{lem T_(g F) =g T_F} implies the following proposition.

\begin{proposition}\label{prop properties of series}
We have for any $\chi\in H$. 
\begin{enumerate}
\item $T_{F_{\chi}} = (T_{F})_{\chi}$.

\item $g\cdot T_{F_{\chi}}  = \chi(g) T_{F_{\chi}} $ for any $g\in G$. 
\item In the notations above, let us have 
$$
T_{F_{\chi}}(t,z) =  \sum a_{n}(t)z^n.
$$
 Then each monomial $z^n = z_1^{n_1} \cdots z_p^{n_p}$ has weight $\chi\in H$. 
\end{enumerate}
\end{proposition}

\begin{proof}
{\bf (1)} First of all for functions $F_1, F_2$ and a constant $c\in \C$ we have
$$
T_{F_1+F_2} = T_{F_1} + T_{F_2}, \quad  T_{c F} = c T_F.
$$
Hence by Lemma \ref{lem T_(g F) =g T_F} we have 
    \begin{align*}
       T_{F_{\chi}} = \frac{1}{|G|} \sum\limits_{g\in G} \chi (g) \,\, [ T_{g^{-1}\cdot F}] = \frac{1}{|G|} \sum\limits_{g\in G} \chi (g) \,\, [g^{-1}\cdot T_{ F}]= (T_{F})_{\chi}.
    \end{align*}

{\bf (2)} The second statement is a consequence of the first one.

{\bf (3)} We have the following equality of series
$$
\chi(g) T_{F_{\chi}}  = \chi(g) \sum a_{n}(t) z^n   = \sum a_{n}(t)(g \cdot z)^n = \sum a_{n}(t) c_n(g) z^n,
$$
where $c_n(g)\in\C$ is a constant depending on $g$. The equality of series implies $c_n(g) = \chi(g)$. The result follows. 
\end{proof}

\begin{remark}
    Usually in the literature one considers a series as in Proposition \ref{prop properties of series} in homogeneous coordinates $(z_i)$ of weight defined by weights of the homogeneous coordinates. Proposition \ref{prop properties of series} implies that our definition of weights leads to the usual one for series. 
\end{remark}

\section{$H$-graded domains}

\subsection{Structure sheaf of a $H$-graded domain}\label{sec def of a G-domain}

We use the notation of Section \ref{sec G-decomposition of F}. More precisely, $G$ is a finite abelian group, $H=\operatorname{Ch} (H)$ is the dual group of $G$ with fixed parity of elements, $V$ is a complex $H$-graded finite-dimensional vector superspace, see  (\ref{eq G-graded superspace}). In the case of $G=\Z_2^n$ we can also assume that $V$ is real.  

\begin{remark}\label{rem H-gradation in V}
Let the group $G$ act on $V$.  By Theorem \ref{theor decomposition superformula} the representation space $V$ has the form $V= \bigoplus\limits_{\chi\in H} V_{\chi}$, where $V_{\chi}$ is the subspace of homogeneous elements of weight $\chi$. We see that two ways are equivalent: to start with a representation of $G$ in $V$; or to start with a $H$-graded vector superspace $V= \bigoplus\limits_{\chi\in H} V_{\chi}$. Note that if a $H$-decomposition of $V$ is given, the action of $G$ in $V$ is defined by formula (\ref{eq action on V by characters}).
\end{remark}

Denote by $\mcF$ the sheaf of complex-valued smooth or holomorphic functions on $V_{\bar 0}$. (In the smooth case, a complex-valued function $f(z)=u(z) + iv(z)$ depending on a complex variable $z=x+iy$ is a smooth function of variables $x$ and $y$.)  Above we saw that for a $G$-invariant open set $W\subset V_{\bar 0}$ we have the following decomposition
$$
\mcF(W) = \bigoplus_{\chi \in H} \mcF_{\chi}(W).
$$

Now we consider the restriction sheaf $\mcS:=\mcF|_U$ for an open set $U\subset V_e$. Let us first recall the construction of a restriction of a sheaf, see \cite[Section 1, Sheaves]{Hartshorne} for details. Let
$$
\iota: V_e\to V_{\bar 0}, \quad (x_i^e) \mapsto (x_i^e, 0),
$$ be the natural inclusion. The sheaf $\mcS$ is the sheaf associated with the following presheaf
\begin{equation}\label{eq direct limit sheaf}
U \supset V \mapsto \lim\limits_{\iota(V)\subset W} \mcF(W),
\end{equation}
where $V\subset U$ is open. 
That is, $\mcS$ is the inverse image sheaf $\iota^{-1} (\mcF)$ of $\mcF$. 

\begin{remark}
  Let $F_i\in \mcF(W_i)$, where $i=1,2$ and $W_i$ are open sets in $V_{\bar 0}$ with $\iota(V)\subset W_i$.  The functions $F_1$ and $F_2$ are called equivalent if there exists an open set $W$ in $V_{\bar 0}$ with $\iota(V)\subset W$ such that $F_1|_{W}= F_2|_{W}$. 
  An element $f$ of the presheaf $\lim\limits_{\iota(V)\subset W} \mcF(W)$ is a class of equivalent functions, and  $\mcS$ is the corresponding sheaf. 
\end{remark}

\begin{example}\label{ex stalk}
    Let $\dim V_e =0$. Then the sheaf (\ref{eq direct limit sheaf}) is a stalk $\mcF_{p}$, where $p=0\in V_e$. (See \cite[Section 1, Sheaves]{Hartshorne}, definition of a stalk, for details.)
\end{example}

The stalk $\mcS_v$, where $v\in U$, is just the stalk $\mcF_{(v,0)}$, see \cite[Section 1, Sheaves]{Hartshorne}. This implies that $\mcS$ is a sheaf of local algebras. Note that we can always assume that $W$ is $G$-invariant. In fact, for any open set $W'$ such that $\iota(V)\subset W'$ we can find a $G$-invariant open set $W=\bigcap_{g\in G} g\cdot W'$ such that $\iota(V) \subset W\subset W'$. Hence, the inverse image sheaf over open sets $W'$ and over $G$-invariant open sets $W$ are the same. Hence, the sheaf $\mcS$ is $H$-graded.

The construction of $\mcS= \mcF|_U$ implies that any section $s\in \mcS(V)$ is defined in a neighborhood $V_v$ of any $v\in V$ by a smooth or holomorphic function $f_{W(V_v)}:W(V_v) \to \mathbb C$, where $W(V_v)$ such that $\iota(V_v) \subset W(V_v)\subset V_{\bar 0}$ is an open neighborhood of the set $V_x$. Clearly, $\{ W(V_v)\}_{v\in V}$ is an open cover of the set $\iota(V)$. We will call the set $\{ f_{W(V_v)}\}$ a {\bf representative data} of the section $s$. If $V_{v_1}\cap V_{v_2}\ne \emptyset$, then there exists an open in $V_{\bar 0}$ set $W\supset V_{v_1}\cap V_{v_2}$ such that 
\begin{equation}\label{eq representatin data}
  f_{W(V_{v_1})}|_W = f_{W(V_{v_2})}|_W.  
\end{equation}
In fact, the functions $f_{W(V_{v_i})}|_W$, $i=1,2$, are both representatives of $s|_{V_{v_1}\cap V_{v_2}}$. Conversely,  the set of functions $\{ f_{W(V_v)}\}$ satisfying (\ref{eq representatin data}) defines a section $s\in \mcS(V)$.

 We can define the value of $s_v\in \mcS_v$ at $v\in U$ by
\begin{equation}\label{eq value of a germ}
s_v(v) = f_{(v,0)} (v,0),
\end{equation}
where $f\in \mcF$ is a local representative.

 We define the following sheaf on $U$ by 
$$
\mcO:= \mcS \otimes_{\C}\bigwedge(V^*_{\bar 1}). 
$$ 
In other words an element $F\in \mcO(V)$ is a Grassmann polynomial with coefficients from $\mcS(V)$. Since $\mcS$ and $\bigwedge(V^*_{\bar 1})$ are naturally $H$-graded, the sheaf $\mcO$ is also $H$-graded. Clearly, the sheaf $\mcO$ is the sheaf of local superalgebras, since $\mcS$ is the sheaf of local algebras. 

Let $V=V_{\bar 0} \oplus V_{\bar 1}  = \bigoplus_{\chi\in H} V_{\chi}$ be an $H$-graded vector space, see Remark \ref{rem H-gradation in V}. Denote $n_{\chi}:= \dim V_{\chi}$.

\begin{definition}\label{def graded domain}
	A $H$-graded domain of dimension $(n_{\chi})_{\chi\in H}$ associated to the $H$-graded superspace $V$ is the locally ringed space $\mcU=(U,\mcO)$, where $U\subset V_e$ is open and $\mcO:= \mcS \otimes_{\C}\bigwedge(V^*_{\bar 1})$. If $(x^g_p,\xi^h_q)$, where $g,h\in H$ and $|g|=\bar 0$ and $|h|=\bar 1$, is a $H$-homogeneous basis of $V^*$, then we call the set $(x^g_p|_{U},\xi^h_q|_{U})$ local graded coordinates of $\mcU$. 
\end{definition}
To simplify the notation, we will write only $(x^g_p,\xi^h_q)$ for local graded coordinates of $\mcU$.

\begin{remark}\label{rem function with 0 Taylor}
	Usually in the literature one considers another structure sheaf of a graded domain, see \cite{Fei,KotovSalnikov,Ji,Jubin}. Let us illustrate the difference. Consider the following function
	$$
	f(x) = \left\{ 
	\begin{array}{ll}
	e^{-\frac{1}{x^2}},& x\ne 0;\\
	0, &x=0.
	\end{array}
	\right. 
	$$
	It is well known that this function has the zero Taylor series centered in $0$. Clearly, this Taylor series does not converge to $f(x)$ in any neighborhood of $0$. 
	
	Now consider a $\Z_2$-graded pure even vector space $V=V_0\oplus V_{1}$, where $V_0\simeq \mathbb C$ and $V_{1}\simeq \mathbb C$. Denote by $z=x+i y$ a  graded coordinate in  $V_1$. Then, according to our definition of the structure sheaf of a graded domain, the class $\overline{f}$ of the function $f(x)$ is not equal to $0$. In fact, $f(x)$ is not equal to $0$ in any open set that contains an open set $U\subset V_0$. If we try to decompose this function into a series around $0$ we will get the trivial series as was mentioned above. 
\end{remark}

\begin{remark}
	Let us discuss another difference between our definition of the structure sheaf of a $H$-graded domain and the definition that is usually used in the literature; see \cite{Fei,KotovSalnikov,Ji,Jubin}. Consider a (super)domain $\mcU = \mathbb C^{1|0}$. Let $(z=x+i y)$ be an even $H$-homogeneous local coordinate of a certain $H$-graded domain $\mcN$ of weight $h\ne e$. Consider the following morphism $\phi:\mcN\to \mcU$ of $\Z_2$-graded domains:
	$$
	\phi^*(w) = z,
	$$
	where $w=w_1+iw_2$ is a local coordinate in $\mcU$. The function $f(w_1)$ defined in Remark \ref{rem function with 0 Taylor} is an element of the structure sheaf of $\mcU$. Then it is natural to assume that 
	$$
\phi^*(f) = f.
	$$ 
	
	If we use approach of other authors, we need to decompose $f(x)$ into the Taylor series at $0$ and we get the following expression
	$$
	f(0) + f'(0) x + \frac12 f''(0)x^2 +  \cdots = 0. 
	$$
	We see that if we consider formal power series in coordinates of weights $h\ne e$, the image of a  set of smooth functions is trivial, even if the morphism is close to the identity. Clearly in the case of other authors, one may lose partially information about the structure of smooth functions. This also may happen for morphisms of $H$-graded domains or $H$-graded manifolds. 
\end{remark}

\begin{remark}
	Using our sheaf $\mcO$ we can construct a sheaf $\mcO'$, which is a structure sheaf of a graded domain in the sense used by another authors. We just decompose any $F\in \mcO(V)$ in a series with respect to the coordinates of the weights $h\ne e$. Note that by Borel's theorem, any formal series is equal to the Taylor series of a certain function. Our idea is to avoid using such series in the smooth case.  
\end{remark}

\begin{proposition}\label{prop F^g(u)=0}
    Let $\mcU=(U,\mcO)$ be a $H$-graded domain, and $F^{\chi}\in \mcO_{\chi}$, where $\chi\ne e$. Then, $F^{\chi} (u) =0$ for any $u\in U$.
\end{proposition}
\begin{proof}
   Without loss of generality, we may assume that $F^{\chi}\in \mcS^{\chi}$. Then for any local representative $F^{\chi}_{\lambda}$ of $F^{\chi}$, we have $$
   \chi(g) F^{\chi}_{\lambda} (u,0) = [g\cdot F^{\chi}_{\lambda}] (u,0) = 
    F^{\chi}_{\lambda} (g^{-1} \cdot u,g^{-1} \cdot 0) = F^{\chi}_{\lambda}(u,0).
   $$
   Hence, $\chi(g)\ne 0$ for some $g$, we have $F^{\chi}_{\lambda}(u,0)=0$. 
\end{proof}

\subsection{Morphisms of $H$-graded domains}\label{sec morphisms of H-graded domains}
Let  $\mcU:=(U,\mcO)$, where $U\subset V_e$ is open,   be a $H$-graded domain corresponding to a representation of  $G$ in $V$. Recall that
$$
\mcO = \mcF|_{U} \otimes_{\C} \bigwedge(V^*_{\bar 1}),
$$
where $\mcF$ is a sheaf of smooth or holomorphic functions on $V_{\bar 0}$, and $\mcF|_{U} = \iota^{-1}(\mcF)$ is the inverse sheaf $\mcF$, where $\iota: U\to V_{\bar 0}$ is the natural inclusion. Note that 
\begin{align*}
\iota^{-1}(\mcF) \otimes_{\C} \bigwedge(V^*_{\bar 1}) = \iota^{-1}(\mcF \otimes_{\C} \bigwedge(V^*_{\bar 1})).
\end{align*}

  We saw that the structure sheaf $\mcO$ is a sheaf of local superalgebras. In fact, for any $v\in U$ the unique maximal ideal $\mathfrak m_v\subset \mcO_v$ is given by
$$
\mathfrak m_v = \{ s\in \mcO_v\,\,|\,\, s(v)=0\},
$$
see (\ref{eq value of a germ}).

\begin{remark}
For superdomains, the following statement is true; see \cite{Leites}. If we have two functions $F_1, F_2 \in \mcO_{\mcV}$ and for any maximal ideal $\mathfrak m_x\in \mcO_{\mcV}$ we have 
$$
F_1 = F_2 \mod \mathfrak m_x^p
$$
for any $x\in \mcV_0$ and any $p\geq 0$, then $\F_1 = \F_2$. This statement is false for our $H$-graded domains.

	 Let $G=K$ and let $V_e=0$, $V_{ab}\simeq \mathbb R$ and $x\in V_{ab}^*$ be a local coordinate. Consider two functions
	\begin{align*}
	\F_1=0,\quad \F_2 = f(x),
	\end{align*}
    where $f(x)$ is from Remark \ref{rem function with 0 Taylor}.
 Clearly, $F_1 = F_2 \mod \mathfrak m_0^p$, where $\mathfrak m_0\in (\mcO_{\mcV})_0$ is the unique maximal ideal. 
\end{remark}

Now, let us take two $H$-graded domains $\mcU=(U,\mcO)$ and $\mcU':=(U',\mcO')$ associated with the $H$-graded vector superspaces  $V$ and $V'$, respectively. We define a morphism of $H$-graded domains
$$
\Phi=(\Phi_0, \Phi^*): \mcU\to \mcU'
$$ 
of type I, where $\Phi_0:U\to U'$ is a smooth or holomorphic map,  as a ringed spaces morphism that preserves the $H$-grading of the structure sheaves and satisfies the following condition.
There exist open sets $W\subset V_{\bar 0}$ and $W'\subset V'_{\bar 0}$, containing  $U$  and $U'$, respectively,  and a morphism of superdomains:
$$
s= (s_0, s^*) : \mcW \to \mcW',
$$ 
where 
$$
\mcW= (W, \mcF_W \otimes_{\C}\bigwedge(V^*_{\bar 1})), \quad \mcW'= (W', \mcF_{W' }\otimes_{\C}\bigwedge(V'^*_{\bar 1}))
$$ 
such that the following diagram is commutative
\begin{equation}\label{eq morphism type II, diagram}
\begin{tikzcd}
\mcU \arrow{r}{\Phi} \arrow[swap]{d}{\iota} & \mcU' \arrow{d}{\iota'} \\
\mcW \arrow{r}{s} & \mcW'
\end{tikzcd}   
\end{equation}
Here 
$
\iota:\mcU \to \mcW$,  $\iota':\mcU' \to \mcW'
$
are the natural embeddings.

Let $(y^g_i,\eta^h_j)$ and $(x^g_p,\xi^h_q)$, where $|g|=\bar 0$ and $|h|=\bar 1$, be local $H$-graded coordinates in $V'$ and $V$, respectively.  Since any morphism of superdomains is determined by the image of local coordinates, see \cite{Leites}, we have
$$
s^*(F(y^g_i,\eta^h_j)) = F(s^*(y^g_i), s^*(\eta^h_j))), \quad F\in \mcO_{\mcW'}.
$$
This implies that $\Phi$ is defined by the image of local coordinates. Indeed, if a local section $f\in \mcO'$ has a representative $F\in \mcO_{\mcW'}$, that is $f = (\iota')^*(F)$, then the commutativity of our diagram implies that 
$$
\Phi^*(f) = \Phi^*(\iota'^*(F)) = \iota^* \circ  s^*(F).
$$
 Summing up if $f = \iota^*(F)$, we have
\begin{equation}\label{eq G domain form local morphism 1}
\Phi^*(f) =  \iota^*(F(s^*(y^g_i), s^*(\eta^h_j))). 
\end{equation}

\begin{remark}
	Formula (\ref{eq G domain form local morphism 1}) is independent on the choice of a representative $F$. Indeed, if $F'$ is another representative of $f$. That is $f = \iota^*(F) = \iota^*(F')$. The commutativity of our diagram implies that  
	$$
	\Phi^*(f) = \iota^*(F(s^*(y^g_i), s^*(\eta^h_j))) =  \iota^*(F'(s^*(y^g_i), s^*(\eta^h_j))).
	$$
\end{remark}

Now assume that $f$ does not have a representative $F$. By our construction $f$ has local representatives $F_{\lambda}$, $\lambda\in \Lambda$. Hence, locally, we have 
$$
\Phi^*(f) = \iota^*(F_{\lambda}(s^*(y^g_i), s^*(\eta^h_j))) .
$$
The corresponding classes
 $$
 \iota^*(F_{\lambda}(s^*(y^g_i), s^*(\eta^h_j))),\quad \iota^*(F_{\mu}(s^*(y^g_i), s^*(\eta^h_j)))
 $$  
 and $\Phi^*(f)$ coincide in the intersection of their domains by commutativity of our diagram. We see that in this case $\Phi^*(f)$ is also defined by images of local coordinates.

\begin{definition}
   We define a morphism $
\Phi=(\Phi_0, \Phi^*): \mcU\to \mcU'$ of $H$-graded domains
of type II, or just a morphism of $H$-graded domains, as a ringed space morphism preserving the $H$-grading of the structure sheaves such that the morphism $\Phi$ locally coincides with a morphism of $H$-graded domains of type I. 
\end{definition}

\begin{lemma}
    Let $\Phi_1: \mcU\to \mcU'$ and $\Phi_2: \mcU'\to \mcU''$, be two morphisms of $H$-graded domains. Then the composition $\Phi_2\circ \Phi_1$ is also a morphism of $H$-graded domains. 
\end{lemma}
 
\begin{proof}
For the morphisms $\Phi_1,\Phi_2$ the diagram (\ref{eq morphism type II, diagram}) is commutative locally. Hence the following diagram
    $$
    \begin{tikzcd}
\mcU \arrow{r}{\Phi_1} \arrow[swap]{d}{\iota} & \mcU' \arrow{r}{\Phi_2}\arrow{d}{\iota'}&  \mcU''\arrow{d}{\iota''}\\
\mcW \arrow{r}{s} & \mcW' \arrow{r}{s'}& \mcW''
\end{tikzcd}   
$$
is also commutative locally (we use obvious notations in the diagram). The result follows. 
\end{proof}

In fact, we have the following important proposition. 

\begin{proposition}\label{prop shapira}
    Let $\mcU=(U, \mcO)$ be  a $H$-graded domain associated to a $H$-graded vector space $V$. Then the natural map
    \begin{equation}
      \Gamma:  \lim\limits_{\iota(U)\subset W}\mcF(W) \to \mcS(U)
    \end{equation}
    is a bijection. 
\end{proposition}

\begin{proof}
We use the idea of \cite[Section: Sheaves on locally compact spaces]{KaS}. For injectivity, assume that $\Gamma(f)=0$. Hence, the stalk $f_x=0$ for any $x\in U$. (We use $\mcF_x=\mcS_x$, $x \in U$.) Therefore,  a representative of $f$ is equal to $0$ in an open neighborhood of $U$, and hence $f=0$ as well. 

Let us prove surjectivity. Let us take $s \in \mcS(U)$ 
with representative data $F_i\in \mcF(W'_i)$, where
$\{W'_i\subset V_{\bar 0}\}$ is an open cover of $U$ in $V_{\bar 0}$, see Section \ref{sec def of a G-domain}.
We need to show that there is an open set $R$ in $V_{\bar 0}$ that contains $U$ and a function $f\in \mcF(R)$ such that $f|_U=s$. In other words, there exists a representative data of $s$ that contains only one open set $R$ and one function $f$.

Without loss of generality, we may assume that the covering $\{W'_i\}$ of $U$ is locally finite. Further, there exists an open, locally finite covering $\{W_i\}$ of $U$ in $V_{\bar 0}$ finer than $\{W'_i\}$ such that $\overline{W}_i\subset W'_i$. Now we define a set
$$
R= \{ x\in \bigcup W_i\,\, |\,\, (F_i)_x = (F_j)_x, \,\, \forall \,\, i,j: \,\, x \in \overline{W}_i\cap \overline{W}_j \}.
$$
Let us prove that $R$ is open in $V_{\bar 0}$. Let $x\in R$. Since the covering $\{W_i\}$ is locally finite, we have only a finite number of sets $W_1, \ldots, W_q$ such that $x\in \overline{W}_i$. We have $(F_i)_x = (F_j)_x$ for any $i, j\in \{1, \ldots q\}$. Therefore there exist an open set $W_x$ such that $(F_i)|_{W_x} = (F_j)|_{W_x}$. We can assume that $W_x$ is covered only by   $W_1, \ldots, W_q$. Hence, $W_x\subset R$. Now we can glue the data $F_i|_{W_i\cap R}$ to get a representative $f$ of $s$.     
\end{proof}

Now we can prove the following theorem. 

\begin{theorem}\label{theor morphism of H-graded is defined by coordinates}
    Let $\mcU =(U, \mcO_{\mcU})$, $\mcU'=(U', \mcO_{\mcU'})$ be two $H$-graded domains corresponding to $H$-graded vector superspaces $V$ and $V'$, respectively,  with graded coordinates $(x^g_p,\xi^h_q)$ and $(y^g_i,\eta^h_j)$, respectively, where $|g|=\bar 0$ and $|h|=\bar 1$, and let the following functions are given
    \begin{equation}\label{eq morphism of H-graded is defined by coordinates}
        \Phi^*(y^g_i) = F^g_i \in (\mcO_{\mcU})_g, \quad  \Phi^*(\eta^h_j) = H^h_j \in (\mcO_{\mcU})_h
    \end{equation}
    such that $(F^e_1(u), \ldots, F^e_{\dim V_e}(u))\in U'$ for any $u\in U$. Then there exists a unique morphism $\Phi:\mcU\to \mcU'$ of $H$-graded domains compatible with (\ref{eq morphism of H-graded is defined by coordinates}). 
\end{theorem}

\begin{proof} 
By Proposition \ref{prop shapira}, for $F^g_i$ and $H^h_j$ there exists an open set $W$ in $V_{\bar 0}$ that contains $U$ and representatives $\tilde F^g_i$ and $\tilde H^h_j$, of $F^g_i$ and $H^h_j$, defined in $W$. Furthermore, by Proposition \ref{prop F^g(u)=0}, $\tilde F^g_i (u) = F^g_i (u) =0$ for any $g\ne e$. Hence, the map $U\ni u \mapsto \tilde F^g_i (u)$ maps $U$ to $U'$. Therefore, the following functions
$$
s^*(y^g_i) = \tilde F^g_i, \quad  s^*(\eta^h_j) = \tilde H^h_j
$$
define a morphism of superdomains $\mcW\to \mcW'$, where $\mcW'$ is some neighborhood of $U'$. Using the morphism $s$ we can  define the morphism $\Phi$ of $H$-graded domains by the following formula
$$
\Phi^*(f(y^g_i,\eta^h_j)) =  \iota^*(\tilde f(s^*(y^g_i), s^*(\eta^h_j))),
$$ 
where $\tilde f$ is a representative of $f$ in some neighborhood of $U$, see Proposition \ref{prop shapira}. Clearly, the definition is independent of the choice of a representative $\tilde f$ and $\Phi$ is compatible with (\ref{eq morphism of H-graded is defined by coordinates}). 
If $t$ is another morphism  of superdomains, which defines a morphism $\Psi$, compatible with (\ref{eq morphism of H-graded is defined by coordinates}), then locally $s=t$. 
\end{proof}

\begin{corollary}
    The definitions of morphisms of type I and II coincide.
\end{corollary}

\subsection{Morphisms of a $H$-graded domain into a superdomain}\label{sec morphism of grdaed to super}

Let $\mcW' = (W',\mcO_{\mcW'})$ be a superdomain, where $W'$ is an open set in $\C^{n}$ and $\mcO_{\mcW'}$ be a sheaf of holomorphic or smooth complex-valued superfunctions. In the case $G=\Z_2^p$, we may also assume that $W'$ is an open set in $\R^{n}$. Now, let us define a morphism 
$$
\Psi=(\Psi_0, \Psi^*) : \mcU \to \mcW',
$$ 
where $\mcU =(U,\mcO)$ is a $H$-graded domain associated with a $H$-graded vector superspace $V$. We assume that $\Psi_0: U\to W'$ is a smooth or holomorphic map and $\Psi^*: \mcO_{\mcW'} \to (\Psi_0)_*\mcO_{\mcU}$ is a morphism of locally ringed spaces such that the following holds. There exists an open subset $W\subset V_{\bar 0}$  and a  morphism of superdomains 
$$
t= (t_0, t^*) : \mcW \to \mcW',
$$ 
where $\mcW= (W, \mcF_W \otimes_{\C}\bigwedge(V^*_{\bar 1}))$,
such that the following diagram is locally commutative
\[
\begin{tikzcd}
& \mcW \arrow[dr,"t"] \\
\mcU \arrow[ur,"\iota"] \arrow[rr,"\Psi"] && \mcW'
\end{tikzcd}.
\]
This definition implies that $\Psi$ is locally defined by the image of local coordinates, since $t$ is given by the images of local coordinates \cite{Leites}.

Now let we have a morphism of superdomains $t= (t_0, t^*) : \mcW \to \mcW'$ and $\iota_0:U\to W$ is the natural inclusion, where $U$ is an open subset in $V_e$. Denote by $\mcU$ the $H$-graded domain induced by the superdomain $ \mcW$. In other words, let $\mcU: = (U, \iota^{-1}(\mcO_{\mcW} ))$.

\begin{proposition}\label{prop morfism H-grdaed to super indused by super}
There exists a morphism $\Psi: \mcU \to \mcW'$ defined by $\Psi_0= t_0\circ \iota_0$ and $\Psi^* (F) = \iota^*\circ t^*(F)$. Since a morphism of superdomains $t$ is determined by images of local coordinates, the same is true for the morphism $\Psi$. 
\end{proposition}

\section{$H$-graded manifolds}

\subsection{Definition of a $H$-graded manifold}\label{sec def H-graded manifold, finite case}

In Sections \ref{sec def of a G-domain} and \ref{sec morphisms of H-graded domains} we defined $H$-graded domains and morphisms between them. Now we define a $H$-graded manifold as a locally ringed space $\mcN=(N,\mcO_{\mcN})$ that is locally isomorphic to a $H$-graded domain satisfying the following condition. Let $\phi_i:\mcN\to \mcU_i$ be two local charts (local isomorphisms from $\mcN$ to a $H$-graded domain), then the composition $\phi_j\circ\phi_i^{-1}: \mcU_i\to \mcU_j $ is a morphism of $H$-graded domains in the sense of Section \ref{sec morphisms of H-graded domains}. A graded manifold morphism is a $H$-graded local morphism, which locally coincides with a morphism of $H$-graded domains. The same for a morphism of a $H$-graded manifold to a supermanifold.

\section{$H$-covering for a superdomain and a supermanifold}

\subsection{$H$-covering of a superdomain}\label{sec H -covering for superdomains}

Let, as above, $G$ be a finite abelian group, $H$ be the dual group of $G$ with fixed parity of elements, that is, with a surjective homomorphism $H\to \Z_2$, $h\mapsto |h|$. Let us take a superdomain $\mcU=(U, \mcO_{\mcU})$  with even and odd coordinates  complex $(x_i,\xi_j)$. (Note that in the smooth case, a complex coordinate $z$ is equal to the sum $u+iv$ of two real coordinates. For $G=\Z_2^p$, any coordinate may be pure real.) Consider the formal even and odd variables $x_i^g$ and $\xi_j^h$, where $g,h\in H$, $|g|=\bar 0$ and $|h|=\bar 1$. The upper index here shows the weight of a coordinate. For example $x_i^e$ is a variable of weight $e\in H$. Denote by $V$ the $H$-graded vector space spanned by $(x_i^g, \xi_j^h)$, where $g,h\in H$. Denote by $\mcV$ the $H$-graded domain with the base space $U$ and the structure sheaf $\mcO_{\mcV}$ associated with $V$ as in Section \ref{sec def of a G-domain}. Note that the $H$-graded domain $\mcV$ and the superdomain $\mcU$ have the same base space $U$ with local coordinates denoted by $x_i^e$ in the graded case  and by $x_i$ in the super case.  Denote by $\mcW'$ the superdomain with coordinates $x_i^g, \xi_j^h$, $g,h\in H$. Again, any of those coordinates have real and imaginary part, and for $G=\Z_2^p$ coordinates may be pure real.

Let us define a morphism $t=(t_0,t^*): \mcW=(W_0,\mcO_{\mcW})\to \mcU$ by the following formula for local coordinates
\begin{align*}
t_0=\id,\quad 
t^*(x_i) = \sum_{|g|=\bar 0} x_i^g,\quad t^*(\xi_j) = \sum_{|h|=\bar 1} \xi_j^h,
\end{align*}
where $W_0:= t_0^{-1} (U) $ is an open set. Now, by Proposition \ref{prop morfism H-grdaed to super indused by super}, we get a morphism $\pp=(\pp_0,\pp^*): \mcV\to \mcU$ defined by
\begin{equation}\label{eq covering map domain}
\pp_0=\id,\quad 
\pp^*(x_i) = \sum_{|g|=\bar 0} x_i^g,\quad \pp^*(\xi_j) = \sum_{|h|=\bar 1} \xi_j^h,
\end{equation}

\begin{definition}
We call the $H$-graded domain $\mcV$ together with the morphism $\pp: \mcV\to \mcU$ given by Formula (\ref{eq covering map domain}) a $H$-covering of a superdomain $\mcU$. 
\end{definition}

\subsection{Universal properties of a $H$-covering of a superdomain}

Now we will show that $\pp: \mcV\to \mcU$ satisfies some universal properties of a covering space. In fact, let us take a $H$-graded domain $\mcV'$ associated with a $H$-graded vector space $V'$, and let $\phi:\mcV' \to \mcU$ be a morphism of a $H$-graded domain into a superdomain. Let us construct a $H$-graded morphism $\Phi: \mcV'\to \mcV$ such that $\phi= \pp\circ \Phi$.  

The morphism $\phi:\mcV' \to \mcU$ is defined by images of local coordinates
$$
\phi^*(x_i) = \sum_{|g|=\bar 0} F_i^{g},\quad \phi^*(\xi_j) = \sum_{|h|=\bar 1} H_j^{h}.
$$
 We put 
\begin{equation}\label{eq lift Phi graded domains}
    \Phi^*(x^g_i) =  F_i^{g}, \quad 
   \Phi^*(\xi^h_j) =  H_j^{h}.
\end{equation}
   
By Theorem \ref{theor morphism of H-graded is defined by coordinates}, see also Proposition \ref{prop F^g(u)=0}, the morphism $\Phi: \mcV'\to \mcV$ is defined.  In addition, $\phi^* (x_i)=[ \Phi^* \circ \pp^*] (x_i)$ and $\phi^* (\xi_j)= [\Phi^* \circ \pp^*] (\xi_j)$. Therefore, $\phi= \pp\circ \Phi$. Since $\Phi$ has to preserve the $H$-grading, Formulas (\ref{eq lift Phi graded domains}) are possible in a unique way. So we proved the following theorem.

\begin{theorem}\label{theo lift of phi mixed to Phi graded domain}
	Let $\phi: \mcV'\to \mcU$ be the morphism from a $H$-domain to a superdomain. Let $\pp:\mcV\to\mcU$ be its $H$-covering. Then there exists a unique morphism of $H$-graded domains $\Phi: \mcV'\to \mcV$ such that the following diagram is commutative:\[
	\begin{tikzcd}
	& \mcV \arrow[dr,"\pp"] \\
	\mcV' \arrow[ur,"\Phi"] \arrow[rr,"\phi"] && \mcU
	\end{tikzcd}
	\]
\end{theorem}

As a consequence of Theorem \ref{theo lift of phi mixed to Phi graded domain} we get the following theorem.

\begin{theorem}\label{theo lift of phi super to Phi graded}
	Let $\psi: \mcU\to \mcU'$ be a morphism of superdomains. Let $\pp:\mcV\to\mcU$ and $\pp':\mcV'\to\mcU'$ be their $H$-coverings, respectively. Then there exists a unique morphism of $H$-graded domains $\Psi: \mcV\to \mcV'$ such that the following diagram is commutative:\[\begin{tikzcd}
	\mcV \arrow{r}{\Psi} \arrow[swap]{d}{\pp} & \mcV' \arrow{d}{\pp'} \\
\mcU \arrow{r}{\psi} & \mcU'
	\end{tikzcd}
	\]
A $H$-covering $\mcV$ of a superdomain $\mcU$ is unique up to isomorphism. 
\end{theorem}

\begin{proof}

To prove this statement, we use Theorem \ref{theo lift of phi mixed to Phi graded domain}. In fact, we put $\phi= \psi\circ \pp$. Then $\Psi$ is a $H$-covering of $\phi$, which exists by Theorem \ref{theo lift of phi mixed to Phi graded domain}.

Now assume that $\mcU$ has two coverings $\pp:\mcV\to \mcU$ and $\pp:\mcV'\to \mcU$. Then by above there exist morphisms $\Psi_1: \mcV\to \mcV'$ and $\Psi_2: \mcV'\to \mcV$ such that the following diagrams are commutative 
\[\begin{tikzcd}
\mcV \arrow{r}{\Psi_1} \arrow[swap]{d}{\pp} & \mcV' \arrow{d}{\pp'} \\
\mcU \arrow{r}{\id} & \mcU
\end{tikzcd};
\quad 
\begin{tikzcd}
\mcV \arrow{r}{\Psi_2} \arrow[swap]{d}{\pp'} & \mcV' \arrow{d}{\pp} \\
\mcU \arrow{r}{\id} & \mcU
\end{tikzcd}
\]
This implies that $\Psi_2\circ \Psi_1 : \mcV\to \mcV$ is a $H$-covering of $\id$. Hence, $\Psi_2\circ \Psi_1=\id$, since $\id$ is also a $H$-covering of $\id$. Similarly, $\Psi_1\circ \Psi_2=\id$. Therefore, a $H$-covering is unique up to isomorphism. 
\end{proof}

Let us give another definition of a $H$-covering. 

\begin{definition}
	Let $\mcU$ be a superdomain. A $H$-covering of $\mcU$ is a $H$-graded domain $\mcV$ together with a morphism $\pp: \mcV\to \mcU$ such that for any  $H$-graded domain $\mcV'$ and a morphism $\phi: \mcV'\to \mcU$ there exists a unique morphism $\Phi: \mcV'\to \mcV$ such that the following diagram is commutative:
		\[
	\begin{tikzcd}
	& \mcV \arrow[dr,"\pp"] \\
	\mcV' \arrow[ur,"\Phi"] \arrow[rr,"\phi"] && \mcU
	\end{tikzcd}
	\]
	In other words a $H$-covering of $\mcU$ is a $H$-graded domain $\mcV$ satisfying universal property.  
\end{definition}

\begin{corollary}
	Two definitions of $G$--coverings of $\mcU$ are equivalent. 
\end{corollary}

\begin{proof}
	We proved that the $H$-covering according to our first definition satisfies the universal property.  We saw, see the proof of  Theorem \ref{theo lift of phi super to Phi graded},  that any object $\mcV$ satisfying the universal property is unique up to isomorphism. The result follows. 
\end{proof}

Let us show that the correspondence $\psi \mapsto \Psi$ is functorial. 
\begin{remark}\label{rem psi to Psi is a functor domains}
	Let $\psi_{12}:\mcU_1\to \mcU_2$ and $\psi_{23}:\mcU_2\to \mcU_3$ be two morphism of superdomains. Denote by $\Psi_{ij}$ the $H$-covering of $\psi_{ij}$. Then the $H$-covering of $\psi_{23} \circ \psi_{12}$ is equal to $\Psi_{23} \circ \Psi_{12}$. In other words, the correspondence $\psi \mapsto \Psi$ is a functor from the category of superdomains to the category of $H$-domains. 
\end{remark}

\subsection{$H$-covering of a supermanifold}

Using results from previous sections, we can construct a $H$-covering $\mcP$ for any supermanifold $\mcM$, where $\mcM$ is holomorphic, real smooth of dimension $2p|2q$ or for $G=\Z_2^n$ it is real smooth or holomorphic. Let us choose an atlas $\{\mcU_i\}$ of $\mcM$. Denote by $\psi_{ij}$ the transition function from $\mcU_j$ to $\mcU_i$. Clearly, these transition functions satisfy the cocycle condition: 
$$
\psi_{ij} \circ \psi_{jk} \circ \psi_{ki} = \id 
$$
in $\mcU_i\cap \mcU_j\cap \mcU_k$. Denote by $\mcV_i$ the $H$-covering of $\mcU_i$ and by $\Psi_{ij}$ the $H$-covering of $\psi_{ij}$. We see that by Remark \ref{rem psi to Psi is a functor domains}, $\{\Psi_{ij}\}$ also satisfies the cocycle condition:
$$
\Psi_{ij} \circ \Psi_{jk} \circ \Psi_{ki} = \id $$ Therefore, the data $\{ \mcV_i\}$ and $\{ \Psi_{ij}\}$ define a $H$-graded manifold denoted by $\mcP$. Denote by $\pp_i: \mcV_i\to \mcU_i$ the covering morphism defined in Section \ref{sec H -covering for superdomains} for any $i$. By construction of the morphisms $\psi_{ij}$ and $\Psi_{ij}$ the following diagram is comitative
$$
\begin{tikzcd}
\mcV_j \arrow{r}{\Psi_{ij}} \arrow[swap]{d}{\pp_j} & \mcV_i \arrow{d}{\pp_i} \\
\mcU_j \arrow{r}{\psi_{ij}} & \mcU_i
\end{tikzcd}.
$$
This means that we can define a morphism $\pp: \mcP\to \mcM$ such that in any chart we have $\pp|_{\mcU_i} = \pp_i$.  Since the diagram above is commutative, the morphisms $\pp_i$ are compatible, hence the global morphism $\pp$ is well defined. 

\begin{definition}
The $H$-graded manifold $\mcP$ constructed for a fixed supermanifold $\mcM$ together with a morphism $\pp: \mcP\to \mcM$ is called a $H$-covering of $\mcM$. 
\end{definition}

Now we show that $\pp: \mcP\to \mcM$ satisfies the universal property. 

\begin{theorem}
	The $H$-covering $\mcP$ of $\mcM$ together with the morphism $\pp: \mcP\to \mcM$ satisfies the following universal property. For any $H$-graded manifold $\mcN$ and a morphism $\phi:\mcN\to \mcM$ there exists a unique morphism $\Phi: \mcN \to \mcP$ such that the following diagram is commutative
\[
	\begin{tikzcd}
	& \mcP \arrow[dr,"\pp"] \\
	\mcN \arrow[ur,"\exists!\Phi"] \arrow[rr,"\phi"] && \mcM
	\end{tikzcd}.
	\]
\end{theorem}

\begin{proof}

Let us take an atlas $\{\mcW_j\}$ of $\mcN$ and an atlas $\{\mcU_i\}$ of $\mcM$. The atlas $\{\mcU_i\}$ generates the atlas $\{\mcV_i\}$ of $\mcP$. Let us show that for any $j$ there exists unique morphism $\Phi_j: \mcW_j \to  \mcP$ of $H$-graded manifolds such that the following diagram is comitative
\[
	\begin{tikzcd}
	& \mcP \arrow[dr,"\pp"] \\
	\mcW_j \arrow[ur,"\exists!\Phi_j"] \arrow[rr,"\phi|_{\mcW_j}"] && \mcM
	\end{tikzcd}.
	\]

By Therem \ref{theo lift of phi mixed to Phi graded domain} for any $i$ there exits unique morphism $\Phi_{ji}:  \mcW_j \to  \mcV_i$ such that the following diagram is comutative
\[
	\begin{tikzcd}
	& \mcV_i \arrow[dr,"\pp"] \\
	\mcW_j \arrow[ur,"\exists!\Phi_{ji}"] \arrow[rr,"\phi|_{\mcW_j}"] && \mcU_i
	\end{tikzcd}.
	\]
This means that the diagram 
\[
	\begin{tikzcd}
	& \mcV_i\cap \mcV_{i'} \arrow[dr,"\pp"] \\
	\mcW_j \arrow[ur,"\Phi_{ji}=\Phi_{ji'}"] \arrow[rr,"\phi|_{\mcW_j}"] && \mcU_i \cap \mcU_{i'} 
	\end{tikzcd}.
	\]
is comitative for $\Phi_{ji}$ and for $\Phi_{ji'}$. Since the graded lift is unique,  we have $\Phi_{ji}= \Phi_{ji'}$ as morphisms $\Phi_{ji}, \Phi_{ji'}:\mcW_j  \to   \mcV_i\cap \mcV_{i'} $. Hence, the required $H$-graded morpism $\Phi_{j}: \mcW_j\to \mcP$ is well defined and unique.

Further, the morphisms $ \Phi_{j}: \mcW_j\to \mcP$ and $\Phi_{j'}: \mcW_{j'}\to \mcP$ both make the following diagram commutative
\[
	\begin{tikzcd}
	& \mcP \arrow[dr,"\pp"] \\
	\mcW_j \cap \mcW_{j'}\arrow[ur,] \arrow[rr,"\phi|_{\mcW_j\cap \mcW_{j'}}"] && \mcM
	\end{tikzcd}.
	\]
Since the $H$-graded lift is unique, we have $ \Phi_{j}|_{\mcW_j \cap \mcW_{j'}} = \Phi_{j'}|_{\mcW_j \cap \mcW_{j'}}$. We put $\Phi|_{\mcW_j} = \Phi_j$. Clearly, $\Phi$ is the required morphism. The proof is complete. 
\end{proof}

\subsection{Examples of $H$-coverings of manifolds and supermanifolds}

\subsubsection{$H$-coverings for usual manifolds}
The paper is devoted to the construction of $H$-coverings for supermaniolds. However, we can apply our construction to the usual manifold. In fact, consider the trivial homomorphism $\phi: H=\Z_2\to 0$. Let us describe the $H$-covering for $\mathbb {CP}^1$.

We fix the standard coordinates $x$ and $y$ for $\mathbb {CP}^1$ with transition functions
$$
y=\frac{1}{x}.
$$
Let us apply the covering projection morphism $\pp$ to our transition functions.
$$
\pp^*(y) = \pp^*\Big(\frac{1}{x}\Big), \quad \text{or} \quad y^{\bar 0} + y^{\bar 1} = \frac{1}{x^{\bar 0} + x^{\bar 1}},
$$
where $\Z_2 = \{ \bar 0, \bar 1\}$. Finally, we obtain the following $H$-graded morphism, given in local graded coordinates by
\begin{align*}
y^{\bar 0} = \frac{ x^{\bar 0}}{(x^{\bar 0})^2 - (x^{\bar 1})^2}, \quad y^{\bar 1} = \frac{ -x^{\bar 1}}{(x^{\bar 0})^2 - (x^{\bar 1})^2}.
\end{align*}

\subsubsection{$H$-coverings for supermanifolds, example}

Let us take $H=\Z_4 =\{\tilde 0,\tilde 1,\tilde 2,\tilde 3\}$ and $\phi: H\to \Z_2$ defined by $\phi(\tilde 1)=\bar 1$. We apply our the covering construction to the projective superspace $\mathbb {CP}^{1|1}$. 
We fix the standard coordinates $(x,\xi)$ and $(y, \eta)$ for $\mathbb {CP}^1$ with transition functions
$$
y=\frac{1}{x}, \quad \eta=\frac{1}{x}\xi.
$$
Applying the covering construction we get
\begin{align*}
&y^{\tilde 0} = \frac{ x^{\tilde 0}}{(x^{\tilde 0})^2 - (x^{\tilde 2})^2}, \quad y^{\tilde 2} = \frac{- x^{\tilde 2}}{(x^{\tilde 0})^2 - (x^{\tilde 2})^2},\\
&\eta^{\tilde 1} = \frac{ x^{\tilde 0}\xi^{\tilde 1}}{(x^{\tilde 0})^2 - (x^{\tilde 2})^2} - \frac{ x^{\tilde 2} \xi^{\tilde 3}}{(x^{\tilde 0})^2 - (x^{\tilde 2})^2},\\
&\eta^{\tilde 3} = \frac{x^{\tilde 0}\xi^{\tilde 3}}{(x^{\tilde 0})^2 - (x^{\tilde 2})^2} - \frac{ x^{\tilde 2} \xi^{\tilde 1}}{(x^{\tilde 0})^2 - (x^{\tilde 2})^2}.
\end{align*}

\section{Graded manifolds and representation theory of finitely generated abelian groups}\label{sec Graded manifolds and representation theory of finitely generated abelian groups}

In this section, we will see that the theory of $H$-graded manifolds is closely related to the representation theory of finitely generated abelian groups $H$.

\subsection{Different approaches to the concept of a $\Z$-graded domain}
In this section, we give a definition of a $H$-graded domain for any finitely generated abelian group $H$ and compare it with some other definitions. For example,  we will see that in the case of a finite abelian group $H$, it is necessary to consider a complex version of the theory.

\subsubsection{A naive approach}\label{sec A naive approach} Let us first consider the case of a $\Z$-graded domain. A naive definition of this domain is as follows. We start with a finite-dimensional $\Z$-graded vector space
$$
V^* = \bigoplus_{k\in \Z} V^*_k,
$$
and choose a basis $
x_1^k, \ldots, x_{\dim V^*_k}^k$ in any $V^*_k$. The number $k\in \Z$ is called the {\it weight} of the variable $x_i^k$. Then we define the weight of a monomial in the natural way. In this case any polynomial or series in our variables $(x^k_i)$ can be decomposed into a (infinite or finite) sum of homogeneous components. 

\subsubsection{An approach via grading operator}\label{sec grading operator} Another approach to the notion of a $\Z$-graded domain, a weight, and a homogeneous component of a function is developed in \cite{Grab}, see also the references cited therein. The graded structure can be encoded by  the following vector field, called a grading operator or a weight vector field  
\begin{equation}\label{eq grading operator}
    v = \sum_{k\in \Z} k x^k \frac{\partial}{\partial x^k}.
\end{equation}
(We have dropped the subscript here.)
We have $v(x^p) = p x^p.$ Now a smooth or holomorphic function $F$ depending on variables $(x^k_i)$ is called homogeneous of weight $p$ if $v (F) = pF$.

\begin{remark}
  The weight $p$ does not need to be an integer, but any real number $p\in \mathbb R$. Furthermore, even in the case where the weights of the local coordinates are integers, a smooth function may have any real weight $r\in \R$. For example, the function $F= (x^k)^{\frac{r}{k}}$ has weight $r$.  
  \end{remark}

\subsubsection{An  approach via character group}\label{sec An  approach via character group} A grading operator $v$ induces a $\Z$-action on $V^*$. Indeed, consider the automorphism $\Gamma: =\exp(v)$. If $F$ is a homogeneous function of weight $p$, for example a local coordinate, we have
\begin{align*}
    \exp(v) (F) = \id (F) + v(F) + \frac{1}{2} v\circ v (F) + \cdots = \\
    F + p F + \frac{1}{2} p^2 F + \cdots = e^p F.
\end{align*}
The $\Z$-action on $V$ is defined as follows
$$
\Z\ni 1 \longmapsto \Gamma, \quad \Z\ni n  \longmapsto \Gamma^n, \quad \Gamma^n (x^p) = e^{np} x^p,
$$
and by duality on $V$.  Since the $\Z$-action in $V$ is defined, it indices the $\Z$-action in the algebra of smooth functions on $V$.  

\begin{remark}
    The $\Z$-action on $V^*$ constructed above is completely reducible. The irreducible components corresponds to the real characters 
    $$
\chi_p:\Z \longrightarrow \mathbb R^{\times}, \quad \Z\ni n \longmapsto \chi_p(n) = (e^{p})^n\in \mathbb R^{\times}. 
$$
   Now we call a function $F$ homogeneous of weight $\chi_p$ if
$$
g \cdot F = \chi_p(g) F, \quad g\in \Z.
$$
In other words, $F$ is homogeneous if it generate an irreducible $\Z$-module. 
\end{remark}

\begin{definition}
  The real (or complex) character group $\operatorname{Ch}_{\mathbb R}(\Z)$  (respectively, \linebreak $\operatorname{Ch}_{\mathbb C}(\Z)$) of $\Z$, also called the real (or complex) Pontryagin dual of $\Z$, is the group of all homomorphisms $
  \chi:\Z \to \mathbb R^{\times}$ (respectively, $\chi:\Z \to \mathbb C^{\times}).
  $
\end{definition}

We see that if a real grading operator $v$ as in Section \ref{sec grading operator}  is defined on $V^*$, then $V^*$ is a real completely reducible $\Z$-representation. Conversely, if $V^*$ is a real completely reducible $\Z$-module with only positive characters, that is $\chi(1)>0$, then we can construct a grading operator $v = \log(\Gamma)$, where $\Gamma$ is the image of $1\in \Z$. If some characters of $V^*$ are negative, we cannot define a corresponding grading operator in general. Indeed, consider an example.  

\begin{example}
    Let $V^*$ be a real completely reducible $\Z$-representation with characters $\chi_1$ and $\chi_{-1}$. In other words,
    $
    \chi_1 (1) =1$, the trivial character, and $\chi_{-1} (1) = -1.
    $
    Let $x^{\chi_1}$ and $x^{\chi_{-1}}$ be graded coordinates. (Assume that $\dim V^*_{\chi_1} = \dim V^*_{\chi_{-1}} = 1$.)  In an additive picture $x^{\chi_1}$ must have weight $0$, since $\chi_1 \chi_1 = \chi_1$, and the weights of $x^{\chi_1}$ and $(x^{\chi_1})^2$ are equal. Let $x^{\chi_{-1}}$ have a certain weight $A\in \R$ with respect to some grading operator. Since $(x^{\chi_{-1}})^2$ has weight $\chi_{1}$, we get $A+A=0$, and hence $A=0$, which is a contradiction, because in the multiplicative picture ${\chi_{-1}} \ne {\chi_{1}}$. 
\end{example}

In summary, the approach through completely reducible real $\Z$-representations is more general than via grading operators with real eigenvalues. In the next sections, we will see that in the case of a finitely generated abelian group, especially a finite abelian group, we need to work over $\C$.

\subsection{Some useful facts from the representation theory of finitely generated abelian groups}\label{sec useful facts from the representation}

It is well known that any finitely generated abelian group $H$ is isomorphic to   
$$
 H\simeq \Z^r \times \Z_{p_1} \times \cdots \times \Z_{p_k}.
 $$
Further we will need the following facts.

 \begin{enumerate}
     \item Any irreducible complex representation of $H$ is $1$-dimensional, and corresponds to a complex character:
     $$
h\cdot v  = \chi(g) v, \quad \chi\in  \operatorname{Ch}_{\mathbb C}(H).
$$
This is a consequence of Schur's lemma.

\item Any real irreducible representation of $H$ is $1$- or $2$-dimensional. 
 A real $1$-dimensional irreducible representation of $H$ corresponds to a real-valued character $\chi:H \to \R^{\times}$. A real $2$-dimensional irreducible representation of $H$ corresponds to a pair of complex conjugate characters $(\chi, \bar \chi)$ that are not real.

\item We have in both real and complex cases
$$
\operatorname{Ch}(G_1\times G_2)  \simeq  \operatorname{Ch}(G_1) \times \operatorname{Ch}(G_2). 
$$

\item  We have
$$
\operatorname{Ch}_{\mathbb C}(\Z_p) \simeq \Z_p, \quad \operatorname{Ch}_{\mathbb R}(\Z_{p}) \subset \{\pm 1\}, \quad \operatorname{Ch}_{\mathbb R}(\Z) \simeq \mathbb R^{\times}, \quad \operatorname{Ch}_{\mathbb C}(\Z) \simeq \mathbb C^{\times}.
$$
For instance, for a finite abelian group, there are sufficient irreducibles only in the complex case. 
 \end{enumerate}

\subsection{$H$-graded domains and manifolds} In this section, we will define a $H$-graded manifold, where $H$ is any finitely generated abelian group, using representation theory. Note that it is not clear how to use a grading operator in this case: we have to replace $k$ in Formula (\ref{eq grading operator}) by an element $h\in H$, which is not defined. 

Let $H$ be a finitely generated abelian group and let $V=V_{\bar 0} \oplus V_{\bar 1}$ be a completely reducible complex finite-dimensional representation of $H$. (We assume here that $H \cdot V_{\bar i} \subset V_{\bar i}$. We can omit this condition considering the group $\tilde H = H\times \Z_2$, where the factor $\Z_2$ defines the parity of an element.) By observations of Section \ref{sec useful facts from the representation} this representation always has the following form
$$
V = \bigoplus_{\chi\in \operatorname{Ch}_{\mathbb C}(H)} V_{\chi},
$$
where $V_{\chi}$ is the sum of equivalent irreducible representations corresponding to the same complex character $\chi$. Note that above we considered a different situation: now $V_{\chi}$ may have non trivial even and odd part simultaneously.

If $W$ is a $H$-invariant open set in $V_{\bar 0}$, then the algebra $\tilde \mcO(W)$ of complex-valued smooth (or holomorphic) superfunctions is a complex $H$-representation. (Note that a function $F(z_1+iz_2)$ is called smooth if it is smooth as a function of two variables $(z_1,z_2)$.) Denote by $\chi_1$ the trivial character of $H$.

\begin{definition}\label{def graded domaim general}
A $H$-graded domain is a ringed space $\mcV :=(U, \mcO)$, where $U\subset V_{\chi_1}\cap V_{\bar 0}$ is an open set and $\mcO$ is the sheaf generated by the presheaf 
$$
U \supset U' \mapsto \lim\limits_{\iota(U')\subset W} \tilde\mcO(W),
$$ 
where $U'\subset U$ is open, $W$ is a $H$-invariant open set in $V_{\bar 0}$ containing $U'$ and $\iota$ is the natural inclusion.  
\end{definition}

\begin{remark}
    We can replace $\tilde\mcO(W)$ by the algebra of  polynomials or formal series. In the case $H= \Z^p \otimes \Z_2^q$ we may assume that $V$ is pure real, since in this case we have enough elements in $\operatorname{Ch}_{\mathbb R}(H)$.
\end{remark}

\begin{remark}
  As we saw in Section \ref{sec An  approach via character group},  we get the previous definition of a $\Z$-graded domain if we assume that all our irreducibles are real $1$-dimensional and correspond to real-valued positive characters. In the case of a finite abelian group  $H$ this assumption is very restrictive, since $\operatorname{Ch}_{\mathbb R}(\Z_{p}) \subset \{\pm 1\}$.  Thus, in the real case for $H$ finite abelian, a meaningful theory is obtained only in the case of a group $H=\Z^{q}_2$. Moreover, typically, one considers the sheaf of series (or polynomials) with respect to coordinates of weight $\chi\ne \chi_1$. In this paper, we also develop a smooth approach. 
\end{remark}

A {\it morphism of graded domains} is an $H$-equivariant morphism of the corresponding ringed spaces with conditions as in Section \ref{sec morphisms of H-graded domains}.  In particular, in coordinates any morphism is given by $H$-homogeneous functions. Now, a $H$-graded manifold for any finitely generated abelian $H$ can be defined as in Section \ref{sec def H-graded manifold, finite case}.

\begin{remark}\label{rem naive definition}
    Let $H$ be a finite abelian group. Consider the naive definition of a $H$-graded domain, see Section \ref{sec A naive approach}. As above, let us choose a basis $x^{g_1}_{i_1}, \ldots, x^{g_k}_{i_k}$ of a real $H$-graded vector space $V^*_{\R} = \bigoplus_{g\in H} V^*_g$, where $g_j\in H$ is the weight of the variable $x^{g_j}_{i_j}$.  Observe that it is not clear how to define a homogeneous smooth function in the variables $(x^{g_j}_{i_j})$ via a grading operator. Further, consider the complexification $V^*_{\C} :=V^*_{\R}\otimes \C$. Now we can define a $H$-action in $V^*_{\C}$ by 
    $$
    h \cdot x^{\chi_{j}}_{i_j}  = \chi_{i} (h) x^{\chi_{j}}_{i_j}
    $$
    and repeat our construction of a $H$-graded domain.
   (Here, we replace $g_i$ by characters using the fact $\operatorname{Ch}_{\mathbb C}(H) \simeq H$.) We see that $V^*_{\R} \subset V^*_{\C}$ is a choice of a real form of a graded domain.  Hence, we obtain the naive definition by assuming that a real form of  $V^*_{\C}$ is chosen and we consider only real polynomials or real series in variables $x^{\chi_{i}}$, $\chi_{i} \ne \chi_{1}$. Note also that the $H$-action does not preserve this real form, and hence a definition of a homogeneous function via characters does not work either.

   More general for complex-analytic graded manifold $\mcM$ in our sense, if in any graded chart a real form as above is fixed and transition functions preserve these real forms, then we can define a real form $\mcM^{\R}$ of $\mcM$. This is a real-analytic $H$-graded manifold, which is usually considered in the literature. We note that, in general, we cannot expect that a real form of a graded manifold exists. 
 Conversely, if we have a real-analytic $H$-graded manifold, then the question of existence of its complexification is related to the corresponding problem in the nongraded case, see \cite{BW,Grauert, Kul}.  
\end{remark}

   \subsection{About a decomposition of a graded function into homogeneous components}

We can give a definition of a homogeneous function for any $H$. Namely, a function $F$ is called homogeneous  of weight $\chi\in \operatorname{Ch}_{\C}(H)$ if 
$$
g \cdot F = \chi(g)F, \quad g\in H. 
$$
Now the question whether a function $F$ can be decomposed into a (finite or infinite) sum of homogeneous components is equivalent to the question whether the subrepresentation generated by $F$ is  completely reducible.  A $H$-representation is always completely reducible in the case where $H$ is finite abelian.


\begin{thebibliography}{999999999}

\bibitem[AKT]{AKT}{\it N Aizawa, Z Kuznetsova, F Toppan}
$\Z_2\times\Z_2$-graded mechanics: the classical theory, The European Physical Journal C, 2020.


\bibitem[AM]{AM}{\it Tomoyuki Arakawa, Anne Moreau}, Arc spaces and vertex algebras, monograph, 2023.

\bibitem[AKY]{AKY}{\it  S. Argenti, M. Kochetov, F. Yasumura},
Group gradings on exceptional simple Lie superalgebras,
Journal of Algebra,
Volume 668,
2025,
Pages 447-490,

    
    \bibitem[Ber]{Ber}	{\it F. A. Berezin.} Introduction to algebra and analysis with anticommuting variables.
	V. P. Palamodov, ed., Moscow State University Press, Moscow, 1983. Expanded transl.
	into English: Introduction to superanalysis. A. A. Kirillov, ed., D. Reidel, Dordrecht, 1987.
	
\bibitem[BHM]{BHM} {\it M. El Baz, Y. Hassouni, F. Madouri} $Z_3$-graded Grassmann variables, parafermions and their coherent states. Physics Letters B
Volume 536, Issues 3-4, 6 June 2002, Pages 321-326 


    
\bibitem[BW]{BW} {\it F. Bruhat, H. Whitney}, Quelques propriétés fondamentales des ensembles analytiques-réels, Comment. Math. Helv. 33 (1959), 132–160. 


	\bibitem[DW1]{Witten not projected} {\it R.~Donagi, E.~Witten}
	Supermoduli Space Is Not Projected. Proc. Symp. Pure Math. 90 (2015) 19-72.
	
	\bibitem[DW2]{Witten Atiyah classes} {\it R.~Donagi, E.~Witten}
	Super Atiyah classes and obstructions to splitting of supermoduli space. Pure and Applied Mathematics Quarterly
	Volume 9 (2013) Number 4, Pages: 739 – 788. 



\bibitem[DEl]{DEl} {\it C. Draper, A. Elduque, C. M. González}, Fine gradings on exeptional Lie superalgebras.  International Journal of Mathematics 2011, 1823-1855. 

\bibitem[El]{El} {\it A. Elduque}, Fine gradings and gradings by root systems on simple Lie algebras, Rev. Mat. Iberoam. 31 (2015), no. 1, 245–266.


\bibitem[GG]{Grab} {\it K. Grabowska, J. Grabowski},
	Graded supermanifolds and homogeneity. 	arXiv:2411.00537.

\bibitem[Gr]{Grauert}  {\it H. Grauert}, On Levi’s problem and the imbedding of real-analytic manifolds, Annals of Mathematics 68 (1958), 460–472.


    
\bibitem[I]{Isaacs} {\it  Isaacs, I.M.} Character Theory of Finite Groups.   Academic Press. ed. 1976. 
	
\bibitem[F]{Fei} {\it M.~Fairon} Introduction to graded geometry.  European Journal of Mathematics, Volume 3, pages 208–222, (2017).  



	
\bibitem[H]{Hartshorne} {\it R.~Hartshorne,} Algebraic Geometry, Graduate Texts in Mathematics (GTM, volume 52), 1977. 

\bibitem[Ish]{Ishii} {\it S. Ishii}, Introduction of arc spaces and the Nash problem, RIMS Kokyu-Roku, 1374, (2004) 40–51.
 
\bibitem[J]{Ji} {\it Shuhan Jiang.} Monoidally graded manifolds, Journal of Geometry and Physics, Volume 183, January 2023.
 
 	
	\bibitem[JKPS]{Jubin} {\it B.~Jubin, A.~Kotov, N.~Poncin, V.~Salnikov.} Differential graded Lie groups and their Lie algebras, Volume 27, pages 497–523, (2022).

    \bibitem[KV]{KV} {\it M. Kapranov, E. Vasserot}, Vertex algebras and the formal loop space, 
 Mathématiques de l'IHÉS 100(1): 209-269.

    

    \bibitem[KaS]{KaS} {\it M. Kashiwara, P. Schapira.} Sheaves on Manifolds, Grundlehren der mathematischen Wissenschaften (GL, volume 292), 1990.

 \bibitem[Ker]{Ker} {\it R. Kerner},  Ternary $Z_2$ and $Z_3$ Graded Algebras and Generalized Color Dynamics,
 Mathematical Structures and Applications, 2018, pages 311-357.
	
	\bibitem[KS]{KotovSalnikov} {\it A.~Kotov, V.~Salnikov.} The category of $\Z$-graded manifolds: what happens if you do not stay positive. 2021. hal-03445075.

\bibitem[KuS]{Kurzweil} {\it Kurzweil H., Stellmacher B.}  The Theory of Finite Groups: An Introduction, 2004. 
	
\bibitem[Kul]{Kul} {R.S. Kulkarni}, On complexifications of differentiable manifolds. Invent Math 44, 49–64 (1978).
	
	\bibitem[L]{Leites} {\it Leites D.A.} Introduction to the theory of supermanifolds. Uspekhi Mat. Nauk, 1980, Volume 35,	Issue 1(211), 3-57.
	

	

\bibitem[Ly]{Ly}{\it Lychagin, V.} Colour calculus and colour quantizations. Acta Appl Math 41, 193–226 (1995).


    \bibitem[RW]{RW} {\it V. Rittenberg, D. Wyler}, Generalized superalgebras, Nuclear Physics B
Volume 139, Issue 3, 10 July 1978, Pages 189-202. 
	
	
	\bibitem[R]{Roytenberg}  {\it Roytenberg, D.} On the structure of graded symplectic supermanifolds and Courant algebroids. Contemp. Math., Vol. 315, Amer. Math. Soc., Providence, RI, (2002).
	
	
	

	\bibitem[Vi1]{Vish} {\it E.~Vishnyakova.}
	Graded manifolds of type $\Delta$ and $n$-fold vector bundles, Letters in Mathematical Physics 109 (2), 2019, 243-293.
	
	
	\bibitem[Vi2]{Vicovering} {\it E.~Vishnyakova.}
Graded covering of a supermanifold, 	arXiv:2212.09558.  

 \bibitem[Vi3]{Vimult} {\it E.~Vishnyakova.}   Multiplicity free covering of a graded manifold, arXiv:2409.02211. 


	\bibitem[Vys]{Vysoky} {\it J.~Vysoky.} Global theory of graded manifolds, Reviews in Mathematical Physics. Vol. 34, No. 10, 2250035 (2022). 

	
\end{thebibliography}
\end{document}